\documentclass{article}
\usepackage{amsthm}
\usepackage{amssymb}
\usepackage{amsfonts}
\usepackage{color}
\usepackage{amsmath}
\begin{document}

\newcommand{\A}{\cal{A}}
\newcommand{\R}{\cal{R}}
\newcommand{\F}{\mathbb{F}}
\newcommand{\C}{\cal{C}}
\newcommand{\N}{\cal{N}}
\newcommand{\cH}{\cal{H}}
\newcommand{\G}{\cal{G}}
\newcommand{\Z}{\cal{Z}}
\newcommand{\T}{\cal{T}}
\newcommand{\Q}{\cal{Q}}
\newcommand{\cS}{\cal{S}}
\newcommand{\W}{\cal{W}}
\newcommand{\E}{\cal{E}}
\newcommand{\D}{\cal{D}}
\newcommand{\cP}{\cal{P}}
\newcommand{\cL}{\cal{L}}
\newcommand{\M}{\cal{M}}
\newcommand{\V}{\cal{V}}
\newcommand{\B}{\cal{B}}
\newcommand{\U}{\cal{U}}
\newcommand{\K}{\cal{K}}
\newcommand{\CC}{{\mathbb C}}
\newcommand{\Aa}{{\mathbb A}}
\newcommand{\FF}{{\mathbb F}}
\newcommand{\FFF}{\overline{\mathbb F}}
\newcommand{\ov}{\overline}
\newcommand{\eps}{\varepsilon}
\newcommand{\wh}[1]{\widehat{#1}}
\newtheorem{theorem}{Theorem}[section]
\newtheorem{proposition}[theorem]{Proposition}
\newtheorem{lemma}[theorem]{Lemma}
\newtheorem{corollary}[theorem]{Corollary}
\newtheorem{remark}[theorem]{Remark}
\newtheorem{definition}[theorem]{Definition}            
\newtheorem{example}[theorem]{Example}            
\newtheorem{conjecture}[theorem]{Conjqecture}
\newtheorem{question}[theorem]{Question}

\title{Order automorphisms of effect algebras\thanks{The author was supported by the grants P1-0288 and N1-0368 from ARIS, Slovenia.}}
\author{Peter \v Semrl\footnote{Institute of Mathematics, Physics, and Mechanics, Jadranska 19, SI-1000 Ljubljana, Slovenia, peter.semrl@fmf.uni-lj.si; Faculty of Mathematics and Physics, University of Ljubljana, Jadranska 19, SI-1000 Ljubljana, Slovenia}
        }

\date{}
\maketitle

\begin{abstract}
An elegant description of the general form of order automorphisms of effect algebras has been known in the complex case. We present a much simpler proof based on the projective geometry which works also in the real case. As an application we classify order isomorphic pairs of matrix intervals and describe the general form of order isomorphisms for any pair of isomorphic matrix intervals. 
\end{abstract}
\maketitle



\section{Introduction}

Let $H$ be a real or complex Hilbert space, $\dim H \ge 2$. We denote by $B(H)$ and $S(H)\subset B(H)$ the algebra of all bounded linear operators on $H$ and the subset of all self-adjoint operators, respectively. An operator $A \in S(H)$ is said to be positive, $A \ge 0$, if $\langle Ax , x \rangle \ge 0$ for every $x \in H$. We write $A > 0$ when $A \ge 0$ and $A$ is invertible. For $A, B \in S(H)$ we define $A \le B \iff B-A \ge 0$ and similarly, $A < B \iff B-A> 0$. It is well-known that if $H$ is a complex Hilbert space and $A \in S(H)$ is positive, then there exists a unique positive square root $A^{1/2}$. It is less known that this is true in the real case as well, see for example \cite{SeT}.
 Let $M,N \subset S(H)$ be any subsets. A map $\phi : M \to N$ is an order isomorphism if it is bijective and for every pair $X,Y \in M$ we have $X \le Y \iff \phi (X) \le \phi (Y)$.

The general form of order automorphisms of $S(H)$ in the complex case was described by Moln\' ar \cite{Mo1}.  Assume that $H$ is a complex Hilbert space with $\dim H \ge 2$
and  $\phi : S(H) \to S(H)$ is an order automorphism. Moln\' ar proved that then there exist an invertible bounded linear or conjugate-linear operator
$T : H \to H$ and $S \in S(H)$ such that
$$
\phi (X) = TXT^\ast + S
$$
for every $X\in S(H)$. We continue with the cones of positive operators
and positive definite operators. For them we use the interval notation, that is,  $$[0, \infty ) = \{ A \in S(H) \, : \, 
A \ge 0 \}$$ and  $$(0, \infty ) = \{ A \in S(H) \, : \, 
A >  0 \}.$$ In \cite{Mo1} it was proved that for every order automorphism $\phi$ of $[0, \infty)$ 
 there exists an invertible bounded linear or conjugate-linear operator
$T : H \to H$ such that
$
\phi (X) = TXT^\ast
$
for every $X\in [0, \infty)$. Similarly, a map
$\phi$ is an order automorphism of $(0, \infty)$ if and only if  there exists an invertible bounded linear or conjugate-linear operator
$T : H \to H$ such that
$
\phi (X) = TXT^\ast
$
for every $X\in (0, \infty)$ \cite{Mol}.

It was much more difficult to describe the general form of automorphisms of $[0,I] = \{ A \in S(H)\, : \, 0 \le A \le I \}$. We first need to introduce a family of bijective monotone increasing functions of the
unit interval onto itself. For every real number $p < 1$ we define such a function $f_p :[0,1] \to [0,1]$ by
$$
f_p (x) = {x \over px + (1-p)}, \ \ \ x \in [0,1].
$$
The following result was proved in \cite{Se0, Se1}.
Let $H$ be a complex Hilbert space,
$\dim H \ge 2$, and $\phi : [0,I] \to [0,I]$ an order automorphism.
Then there exist real numbers $p,q \in (-\infty , 1)$ and a bijective linear or conjugate-linear bounded operator
$T : H \to H$ with $\| T \| \le 1$ such that 
\begin{equation}\label{misteri}
\phi (X) = f_q \left( \left( f_p (TT^* ) \right) ^{-1/2} f_p (TXT^*) \left( f_p (TT^* ) \right) ^{-1/2} \right), \ \ \
X \in [0,I].
\end{equation}
This is quite a mysterious formula. For example, if we take a compositum of two maps of the form (\ref{misteri}), then the new map is again an order automorphism of $[0,I]$, and therefore it must be of the same form. But it seems to be impossible to deduce this fact directly from (\ref{misteri}). 

The final step was to develop the general theory of order isomorphisms of operator intervals, see \cite{Se6}. 

In the mathematical foundations of quantum mechanics the operator interval $[0,I]$ is called the effect algebra on a Hilbert space $H$. It plays the central role in Ludwig's axiomatic formulation of quantum mechanics, see \cite{BLPY}, \cite{Dav}, \cite{Kraus}, \cite{LudI}, and \cite{LudII}. Symmetries of effect algebras, that is, automorphisms of effect algebras were studied a lot, see \cite{DoM}, \cite{GeS}, \cite{HKX}, \cite{Moln1}, \cite{Mo2}, \cite{MoK}, \cite{MoN}, \cite{MoP}, \cite{MoS}, \cite{Se0}, \cite{Se1}, \cite{Se4}, and the references therein. One assumes that $\phi : [0,I] \to [0,I]$ is a bijective map which preservers certain operations and/or relations that are relavant in mathematical foundations of quantum mechanics and then the goal is to describe the general form of such maps. In most cases the description of such maps is very simple. Therefore the converse statement in such cases is trivial and it is easy to deduce that the set of symmetries is a group. This is not the case with the description of order automorphisms of $[0,I]$ given by (\ref{misteri}). 

This difficulty was resolved by the following characterization of order automorphisms of effect algebras \cite[Theorem 7.3]{MS} which is much nicer than the one given by the formula (\ref{misteri}).
Let $H$ be a complex Hilbert space, $\dim H \ge 2$, and $\phi : [0, I] \to [0, I]$ an order automorphism.
Then there exists a bijective linear or conjugate-linear bounded operator
$T: H\to H$ that is unique up to a multiplication with a complex number of modulus one, such that 
\begin{equation}\label{nice}
\phi (X) = \phi_T (X) = T \left( X (T^\ast T -I) +I \right)^{-1} X T^\ast
\end{equation}
for every
$X \in [0,I]$.

Our paper was motivated by the following natural question: does the above theorem holds true also in the real case? Of course, in the real case we need to replace ``linear or conjugate-linear" by ``linear" and ``a complex number of modulus one" by 
``$\pm 1$". Our main result gives the affirmative answer.
If we restrict to the finite-dimensional case then real symmetric $n \times n$ matrices equipped with the partial order $\le$ are at least as important for applications as complex hermitian matrices. Thus, we believe that our question is interesting at least in the finite-dimensional case. This might not be true in the infinite-dimensional case where self-adjoint operators on complex Hilbert spaces  equipped with the partial order $\le$ are of basic importance in mathematical foundations of quantum mechanics while self-adjoint operators on infinite-dimensional real Hilbert spaces seem to be of limited interest. The reason that we will not restrict to the finite-dimensional case is that our approach gives a new much simpler proof of the main theorem also in the complex case.

It should be mentioned that the description of order isomorphisms between two effect algebras has been obtained in a much more general setting of atomic JBW-algebras \cite{RW,VIR}. But again, the description of such maps given in these two papers is far more complicated than (\ref{nice}). The emphasize in this paper is on the elegant formula (\ref{nice}) and a simple proof.

The proof of \cite[Theorem 7.3]{MS}
depends heavily on the fact that the underlying Hilbert space is over the complex field. Let us explain this briefly. Assume that $H$ is a complex Hilbert spce. Let $U \subset S(H)$ be an operator domain, that is, an open connected subset of $S(H)$.
We denote by $\Pi(H)$ the set of all operators $X+iY \in B(H)$, where $X,Y \in S(H)$ and $Y$ is positive and invertible.
The Loewner's theorem for maps on operator domains proved in \cite{MS} states that a map $\phi : U \to S(H)$ is a local order isomorphism if and only if $\phi$ 
 has a unique continuous extension to $U \cup \Pi (H)$ that maps $\Pi (H)$ biholomorphically onto itself. It is possible to use results from infinite dimensional holomorfy \cite{Har} to obtain a nice formula that describes biholomorphic automorphisms of $\Pi (H)$. The above nice description of order automorphisms of $[0,I]$  is a rather simple corollary. Obviously, this method does not work in the real case.

To answer our question we need a new approach. It is based on the projective geometry and works in both the complex and the real case.
At first glance it seems impossible to deduce (\ref{nice}) from the fundamental theorem of projective geometry. Namely, in the conclusion of the fundamental theorem of projective geometry a semilinear map appears while the map $\phi$ in the formula (\ref{nice}) is far from being additive. But if we restrict our attention just to projections then we shall see that $\phi$ in (\ref{nice}) is induced by a linear map in the real case and by a linear or conjugate-linear map in the complex case. Indeed, let $P \in [0,I]$ be any projection. First we need to verify that $P(T^\ast T -I) +I$ is invertible and $\phi (P)$ is self-adjoint. Assume for a moment that this is true (a much stronger statement will be proved later, see the first two paragraphs of the proof of Theorem \ref{reacom}).     Let us verify that $\phi (P)$ is an idempotent. We first claim that for every $A  \in S(H)$, $A > -I$ (note that $A = T^\ast T -I$ satisfies this condition), we have $(P-I)(PA+I)^{-1}P= 0$. All we need to do to verify this claim is to represent all operators as $2\times 2$ operator matrices with respect to the orthogonal direct sum decomposition $H = {\rm Im}\, P \oplus {\rm Ker}\, P$, that is,
$$
P= \left[ \begin{matrix} I & 0 \cr 0 & 0 \cr \end{matrix} \right] , \ \ \  P-I = \left[ \begin{matrix} 0 & 0 \cr 0 & -I \cr \end{matrix} \right] ,
$$  
and
$$
PA  + I = \left[ \begin{matrix} A_1 + I & A_2 \cr 0 & I \cr \end{matrix} \right],
$$
which yields that $(PA+I)^{-1}$ is of the form
$$
(PA+I)^{-1} =  \left[ \begin{matrix} * & * \cr 0 & I \cr \end{matrix} \right]. 
$$
The desired equality $(P-I)(PA+I)^{-1}P= 0$ can now be verified by a straightforward computation. It follows that
$$
\phi (P)^2 =  T \left( P (T^\ast T -I) +I \right)^{-1} (P T^\ast  T) \left( P (T^\ast T -I) +I \right)^{-1} P T^\ast
$$
$$
=  T \left( P T^\ast T -P +I \right)^{-1}( (P T^\ast  T -P +I) +(P-I)) \left( P (T^\ast T -I) +I \right)^{-1} P T^\ast
$$
$$
=  T  \left( P (T^\ast T -I) +I \right)^{-1} P T^\ast = \phi (P),
$$
as desired. Hence, for every projection $P \in [0,I]$ the operator $\phi (P)$ is a projection, too.
Let us next consider the kernel of $\phi (P)$. Let $x \in H$. Since $T$ and $ \left( P (T^\ast T -I) +I \right)^{-1}$ are invertible we have $x \in {\rm Ker}\, \phi (P)$ if and only if $PT^\ast x = 0$. Equivalently, the kernel of $\phi (P)$ is equal to $\left( T^\ast \right)^{-1} ({\rm Ker}\, P)$. Therefore,
\begin{equation}\label{fiesa}
{\rm Im}\, \phi (P) = ({\rm Ker}\, \phi (P))^\perp = \left( \left( T^{-1} \right)^\ast ({\rm Ker}\, P) \right)^\perp = T({\rm Im}\, P).
\end{equation}
Hence, if we restrict the map $\phi$ given by (\ref{nice}) to the set of all projections and if we, as usual, identify projections with their images, then this restriction is induced by the linear (or conjugate-linear) bounded invertible operator $T$. This observation indicates that trying to solve our problem by using the fundamental theorem of projective geometry makes sense. 
There is one limitation of such an approach, that is, the special case when $\dim H = 2$ needs to be treated separately because the assumption that $\dim H \ge 3$ is essential in the fundamental theorem of projective geometry.

In the next section we will present some lemmas that will be needed for the proof of our main theorem. The proof of the main theorem in the case that $\dim H \ge 3$ will be given in the third section and in the fourth section we will treat  the two-dimensional case. The last section will be devoted to order isomorphisms of matrix intervals. The question is which matrix intervals are order isomorphic and in the case when two matrix intervals are order isomorphic we would like to have a description of all order isomorphisms. We will treat only the real case since the complex case has been considered already in \cite{Se6}. Here we will use a completely different approach which will clearly show that our main result is the cruical step in developing the theory of order isomorphisms of matrix or operator intervals.

\section{Preliminary results}

Let $H$ be a real or complex Hilbert space. The following notion is an extension of the definition of the strength of an effect along a ray given in \cite{BG}. Let $A \in S(H)$ be positive and $P: H \to H$ a projection of rank one. Then the strength $\alpha (A,P)$ of $A$ along $P$ is defined as
$$
\alpha (A,P) = \sup \{  t \in \mathbb{R} \,  : \, tP \le A \} = \max \{ t \in \mathbb{R} \, : \, tP \le A \}.
$$
Clearly, $0 \le \alpha (A,P) \le \| A \|$. 

The next lemma and its proof are slight modifications of \cite[Theorem 1]{BG}. Recall that if $Q$ is a rank one projection whose image is spanned by a unit vector $u$ and $z$ is any unit vector, then $\langle Qz,z \rangle = | \langle u,z \rangle |^2$. This is well-known and easy to verify.
\begin{lemma}\label{less}
Let $A,B \in S(H)$ be positive. Then the following two statements are equivalent.
\begin{enumerate}
\item $A \le B$.
\item For every rank one projection $P : H \to H$ we have $\alpha (A,P) \le \alpha (B,P)$.
\end{enumerate}
\end{lemma}

\begin{proof}
The implication $(1) \Rightarrow (2)$ is trivial. To prove the converse assume that the second condition is fulfilled and take any unit vector $x \in H$. If $\langle Ax , x \rangle = 0$ then clearly,  $\langle Ax , x \rangle \le  \langle Bx , x \rangle$. So, assume that  $\langle Ax , x \rangle > 0$. Let $Q$ be the rank one projection whose image is spanned by $Ax$ and set 
$$
t = { \| Ax\|^2 \over \langle Ax , x \rangle } .
$$
We will prove that 
\begin{equation}\label{hit}
tQ \le A .
\end{equation}
Assume for a moment that this has been already proved. Then, by our assumption, $tQ \le B$. It follows (see the remark before the formulation of our lemma) that
$$
\langle Ax , x \rangle =  { \| Ax\|^2 \over \langle Ax , x \rangle } \, \left \langle {Ax \over \| Ax \|} , x \right\rangle^2 = t \langle Qx , x \rangle \le \langle Bx , x \rangle ,
$$
as desired.

It remains to prove (\ref{hit}). To this end chose any unit vector $z \in H$. Then
$$
 t \langle Qz , z \rangle =  { \| Ax\|^2 \over \langle Ax , x \rangle } \, \left| \left \langle {Ax \over \| Ax \|} , z \right\rangle \right|^2 
=  { 1 \over \langle Ax , x \rangle } \, \left| \left \langle Ax , z \right\rangle \right|^2 ,
$$
and because $A \ge 0$ and $ \langle Ax , x \rangle >0$, the Cauchy-Schwarz's inequality yields 
$$
 t \langle Qz , z \rangle \le  { 1 \over \langle Ax , x \rangle } \langle Ax , x \rangle \langle Az , z \rangle =  \langle Az , z \rangle .
$$
\end{proof}

As a trivial consequence we see that if $A,B \in S(H)$ are positive, then $A =B$ if and only if for 
 every rank one projection $P : H \to H$ we have $\alpha (A, P) = \alpha (B,P)$. This further implies that if $A \in S(H)$ is positive and $A$ is not the zero operator or a rank one operator, then there exist two different rank one projections $P,Q$ and two real numbers $t,s >0$ such that $tP \le A$ and $sQ \le A$ (of course, in the complex case we see this directly from the spectral theorem for bounded self-adjoint operators).

\begin{lemma}\label{projtwo}
Let $H$ be a real or complex Hilbert space, $\dim H \ge 2$, and $P,Q$ a pair of projections of rank one with $P\not=Q$. Assume that $A \in S(H)$ is an operator such that $P \le A \le I$ and $Q \le A \le I$. We denote by $K$ the two-dimensional subspace of $H$ spanned by the images of $P$ and $Q$. Then the operator matrix representation of $A$ with respect to the orthogonal direct sum decomposition $H = K \oplus K^\perp$ is of the form
$$
A = \left[ \begin{matrix} I_K & 0 \cr 0 & B \cr \end{matrix} \right],
$$
where $I_K$ stands for the identity operator on $K$, and $B :  K^\perp \to K^\perp$ is a positive operator with $B \le I_{K^\perp}$. 
\end{lemma}

\begin{proof}
Let $x\in H$ be a unit vector that spans the image of $P$. We have
$$
1 = \langle Px , x \rangle \le  \langle Ax , x \rangle \le \| Ax\| \, \|x \| \le 1,
$$
which means that all inequalities are actually equalities and then by the Cauchy-Schwarz theorem we see that $Ax$ and $x$ are linearly depenedent. Because $Ax$ is a unit vector and $A$ is positive we have $Ax=x$. Similarly, $Ay = y$ for a unit vector $y$ that spans the image of $Q$. It follows that $Au=u$ for every $u \in K$. Using the fact that $I-A$ has the operator matrix representation
$$
I - A = \left[ \begin{matrix} 0 & * \cr 0 & * \cr \end{matrix} \right]
$$ 
and $I-A \in S(H)$ we immediately see that
$$
A = \left[ \begin{matrix} I_K & 0 \cr 0 & B \cr \end{matrix} \right],
$$
for some linear self-adjoint bounded operator $B :  K^\perp \to K^\perp$. From $0 \le A \le I$ we conclude that $B$ has the desired properties.
\end{proof}

\begin{lemma}\label{rkone}
Let $H$ be a real or complex Hilbert space, $\dim H \ge 2$, and $A,B \in [0,I]$ with $A \le B$. Then the following are equivalent.
\begin{itemize}
\item There exists a real number $t \ge 0$ and a rank one projection $P$ such that $B = A +tP$.
\item For every pair $C,D \in [0,I]$ satisfying $A \le C,D \le B$ we have $C \le D$ or $D \le C$.
\end{itemize}
\end{lemma}

\begin{proof}
Let us first assume that $A,B \in [0,I]$ and $B = A +tP$ for some real number $t \ge 0$ and a rank one projection $P$.
It is rather trivial to verify that if $C,D \in [0,I]$ satisfy $A \le C,D \le B$, then $C = A+ pP$ and $D= A+qP$ for some real numbers $p,q$ belonging to the interval $[0,t]$. It is clear that then $C$ and $D$ are comparable.

To prove the other direction we assume that $\dim {\rm Im}\, (B-A) \ge 2$. Then by the remark following Lemma \ref{less}, we see that  there exist two different rank one projections $P,Q$ and two positive real numbers $p,q$ such that $C = A + pP \le B$ and $D = A+sQ \le B$. Since obviously $A \le C,D \le B$, and neither $C \le D$ nor $D\le C$, the proof of the lemma is completed. 
\end{proof}

\begin{lemma}\label{mikmik}
Let $H$ be a real or complex Hilbert space, $\dim H \ge 2$, $R$ any projection of rank one, and $s$ a real number, $1/2 < s < 1$. Then there exists a pair of orthogonal rank one projections $P,Q$ such that for every real number $p$ we have
$$
pR \le (1/2) P + Q \iff p \le s.
$$
\end{lemma}

\begin{proof}
The numerical range of $R$ is the closed unit interval $[0,1]$. It follows from $1/2 < s < 1$ that 
$$
0 < { 1- s \over s} < 1.
$$
Hence, we can find a unit vector $x \in H$ such that $\langle Rx , x \rangle = (1-s)/s$. We further find a unit vector $y$ that is orthogonal to $x$ such that ${\rm Im}\, R \subset {\rm span}\, \{ x,y \}$. Then with respect to the orthogonal direct sum decomposition $H = {\rm span} \, \{ x,y \} \oplus \{ x,y \}^\perp$ the rank one projection $R$ has the matrix representation
$$
R =  \left[ \begin{matrix} R_1 & 0 \cr 0 & 0 \cr \end{matrix} \right]
$$
with $R_1$ being a $2\times 2$ rank one projection of the form
$$
R_1 =  \left[ \begin{matrix} {1-s \over s} & * \cr * & 1-  {1-s \over s} \cr \end{matrix} \right] .
$$
It is clear that from now on we can work with $2\times 2$ matrices. Thus $R$ equals the above projection $R_1$ and set
$$
P =  \left[ \begin{matrix} 1 & 0 \cr 0 & 0 \cr \end{matrix} \right] \ \ \ {\rm and} \ \ \ Q = \left[ \begin{matrix} 0 & 0 \cr 0 & 1 \cr \end{matrix} \right] .
$$

Obviously, if $C$ is a $2\times 2$ self-adjoint matrix of rank at most one, then $C \le I$ if and only if ${\rm tr}\, C \le 1$. Hence, for a real number $p$ we have
$$
pR \le (1/2)P + Q =  \left[ \begin{matrix} 1/2 & 0 \cr 0 & 1 \cr \end{matrix} \right]
$$
if and only if
$$
 \left[ \begin{matrix} \sqrt{2}  & 0 \cr 0 & 1 \cr \end{matrix} \right] \, \left( p \left[ \begin{matrix} {1-s \over s} & * \cr * & 1-  {1-s \over s} \cr \end{matrix} \right] \right) \,  \left[ \begin{matrix} \sqrt{2}  & 0 \cr 0 & 1 \cr \end{matrix} \right]
$$
$$
\le  \left[ \begin{matrix} \sqrt{2}  & 0 \cr 0 & 1 \cr \end{matrix} \right] \,              \left[ \begin{matrix} 1/2 & 0 \cr 0 & 1 \cr \end{matrix} \right]          \, \left[ \begin{matrix} \sqrt{2}  & 0 \cr 0 & 1 \cr \end{matrix} \right] ,
$$
which is equivalent to
$$
{\rm tr}\, \left( \left[ \begin{matrix} \sqrt{2}  & 0 \cr 0 & 1 \cr \end{matrix} \right] \, \left( p \left[ \begin{matrix} {1-s \over s} & * \cr * & 1-  {1-s \over s} \cr \end{matrix} \right] \right) \,  \left[ \begin{matrix} \sqrt{2}  & 0 \cr 0 & 1 \cr \end{matrix} \right] \right) \le 1.
$$
Therefore we have $pR \le (1/2)P + Q$ if and only if
$$
p \left( { 2 (1-s) \over s} + 1-  {1-s \over s} \right) \le 1
$$
which happens if and only if $p \le s$.
\end{proof}

The following lemma and its proof are very similar to \cite[Lemma 3.44]{MS1}.

\begin{lemma}\label{chrono}
Let $H$ be a two-dimensional real or complex Hilbert space, $P$ a rank one projection on $H$, and $A \in S(H)$. Then the following are equivalent.
\begin{itemize}
\item There exists a rank one projection $Q$ such that ${\rm tr}\, (PQ)= 1/2$ and $A = (1/3)Q + (I-Q)$. 
\item There exist nonnegative real numbers $t_1 , t_2 , t_3$ and rank one projections $R_1 , R_2 , R_3$ such that
$$
A = (1/2)P + t_1 R_1 , \ \ \ A = (1/2)(I-P) + t_2 R_2 , \ \ \ {\rm and} \ \ \ A = I- t_3 R_3.  
$$ 
\end{itemize}  
\end{lemma}

\begin{proof}
We identify operators on $H$ with $2\times 2$ matrices. Assume that $A$ and $P$ satisfy the first condition. Then, after applying an appropriate unitary similarity we may assume that
$$
Q= \left[ \begin{matrix} 1& 0 \cr 0 & 0 \cr \end{matrix} \right] \ \ \ {\rm and} \ \ \ A = \left[ \begin{matrix} 1/3 & 0 \cr 0 & 1 \cr \end{matrix} \right] .
$$
From ${\rm tr}\, (PQ) = 1/2$ we conclude that both diagonal entries of $P$ are equal to $1/2$. From $P\in S(H)$ and $\det P =0$ we see that the off-diagonal entries of $P$ have absolute value $1/2$. After applying yet another unitary similarity we may assume with no loss of generality that
$$
P= \left[ \begin{matrix} 1/2& 1/2 \cr 1/2 & 1/2 \cr \end{matrix} \right] .
$$
Then
$$
A - (1/2)P = \left[ \begin{matrix} 1/12& -1/4 \cr -1/4 & 3/4 \cr \end{matrix} \right]
$$
is obviously a rank one positive matrix. Similarly, both $A - (1/2) (I-P)$ and $I-A$ are positive matrices of rank one. This completes the proof in one direction.

Assume next that the second condition is fulfilled. With no loss of generality we can assume that 
$$
P= \left[ \begin{matrix} 1& 0 \cr 0 & 0 \cr \end{matrix} \right] .
$$
We have
$$
A = \left[ \begin{matrix} a& \alpha \cr \overline{\alpha} & b \cr \end{matrix} \right]
$$
for some real numbers $a$ and $b$ and a real or complex number $\alpha$ (of course, in the real case $\overline{\alpha} = \alpha$). By the second condition we have $\det (A - (1/2)P) = \det (A - (1/2) (I-P))=0$ which gives
$$
(a - 1/2) b = |\alpha |^2 \ \ \ {\rm and} \ \ \ a (b- (1/2)) =  |\alpha |^2 .
$$ 
It follows that $a=b$ and 
\begin{equation}\label{hoja}
a^2 -(1/2)a = |\alpha |^2 .
\end{equation} 
From $\det (I-A) = 0$ we get $(1-a)^2 = |\alpha |^2$ which together with (\ref{hoja}) yields $a=2/3$ and $\alpha = (1/3)z$ where in the complex case $z$ is a complex number of modulus one while in the real case $z \in \{ -1 , 1 \}$. Thus,
$$
A = \left[ \begin{matrix} 2/3 & (1/3) z \cr (1/3)\overline{z} & 2/3 \cr \end{matrix} \right] = (1/3)Q + (I-Q),
$$
where 
$$
Q =  \left[ \begin{matrix} 1/2 & -(1/2) z \cr - (1/2)\overline{z} & 1/2 \cr \end{matrix} \right]
$$
is a rank one projection satisfying ${\rm tr}\, (PQ) = 1/2$.
\end{proof}

In the rest of this section we will restrict our attention to real matrices.  
For any positive integer $n$ we denote by $S_n$ the set of all $n \times n$ real symmetric matrices. Let $[0,I] \subset S_2$ be the set of all $2\times 2$ real effects.
By $\mathcal{D}$ we denote the set of all diagonal effects,
$$
\mathcal{D} = \left\{ \left[ \begin{matrix} s& 0 \cr 0 & t \cr \end{matrix} \right] \, : \, 0 \le s,t \le 1 \right\}.
$$
For every $X \in \mathcal{D}$, $X = {\rm diag}\, (s,t)$, we will denote by $X^\sharp$ the effect
\begin{equation}\label{sharp}
X^\sharp = s  \left[ \begin{matrix} 1/2 & 1/2 \cr 1/2 & 1/2 \cr \end{matrix} \right] + t  \left[ \begin{matrix} 1/2 & -1/2 \cr -1/2 & 1/2 \cr \end{matrix} \right].
\end{equation}

Let $A \in [0,I] \subset S_2$. We set $$
\mathcal{D}_A = \{ X  \in \mathcal{D} \, : \, X \le A \}.
$$
We further denote by $\mathcal{R} \subset [0,I]$ the set of all rank one effects with nonzero off-diagonal entries.

\begin{lemma}\label{max}
Let $A,B \in [0,I] \subset S_2$. The following are equivalent.
\begin{itemize}
\item $\mathcal{D}_A = \mathcal{D}_B$.
\item $A= B$ or $A,B \in \mathcal{R} \cup \{ 0 \}$ or $A = JBJ$,
where
$$
J = \left[ \begin{matrix} 1& 0 \cr 0 & -1 \cr \end{matrix} \right] .
$$
\end{itemize}
\end{lemma}

\begin{proof}
We first note that a $2 \times 2$ symmetric matrix is positive if and only if both diagonal entries are nonengative and its determinant is nonnegative. Indeed, assume first $A \in S_2$ is positive. Let $\{ e_1 , e_2 \}$ be the standard basis of $\mathbb{R}^2$. Then $\langle Ae_1 , e_1 \rangle \ge 0$ and  $\langle Ae_2 , e_2 \rangle \ge 0$ and this tells us that the diagonal entries of $A$ are nonnegative. Since both eigenvalues of $A$ are nonnegative we conclude that $\det A \ge 0$. Assume next that  both diagonal entries of $A$ are nonnegative and $\det A \ge 0$. From the nonnegativity of the determinant we infer that both eigenvalues of $A$ are nonnegative or both are nonpositive and at least one is negative. In the second case the trace of $A$ would be negative contradicting our assumption. Thus, both eigenvalues are nonegative and therefore $A \ge 0$, as desired.

In the next step we will describe the set of maximal elements of $\mathcal{D}_A$ for any $A\in [0,I]$. Let us start with the simplest case that $A \in \mathcal{D}$. Then obviously, the set of maximal elements of $\mathcal{D}_A$ is the singleton $\{ A \}$. Thus we assume from now on that the off-diagonal entries of $A$ are nonzero,
$$
A = \left[ \begin{matrix} t& u \cr u & s \cr \end{matrix} \right] 
$$
with $u \not= 0$. Since $A \ge 0$ we have $ts - u^2 \ge 0$. By the first paragraph of the proof we know that for a pair of real numbers $p,q \in [0,1]$ we have
$$
\left[ \begin{matrix} p& 0 \cr 0 & q \cr \end{matrix} \right] \le \left[ \begin{matrix} t& u \cr u & s \cr \end{matrix} \right]
$$
if and only if $0 \le p \le t$ and $0 \le q \le s$ and 
\begin{equation}\label{detpos}
(t-p)(s-q) - u^2 \ge 0.
\end{equation}
The first case that we will treat is that $A$ is of rank one and has nonzero off-diagonal entries. Then $ts = u^2 \not= 0$. If at least one of the nonnegative numbers $p,q$, $p \le t$, $q \le s$, is positive then
$$
(t-p)(s-q) - u^2 < ts - u^2 = 0 .
$$
It follows that  $\mathcal{D}_A = \{ 0 \}$. 

It remains to consider the case that 
$$A =  \left[ \begin{matrix} t& u \cr u & s \cr \end{matrix} \right]  \in [0,I]
$$ 
is not diagonal and is of rank two. We claim that the set $\mathcal{S}$ of all maximal elements of
$\mathcal{D}_A$ is equal to
$$
\mathcal{S} = \left\{    \left[ \begin{matrix} p& 0 \cr 0 & q \cr \end{matrix} \right] \, : \, 0\le p \le t \ \, {\rm and} \ \, 0 \le q \le s \ \, {\rm and} \ \, (t-p)(s-q) = u^2 
\right\}. 
$$
Let ${\rm diag}\, (p,q) \in \mathcal{S}$. Then clearly ${\rm diag}\, (p,q) \le A$. Take any ${\rm diag}\, (p',q') \in \mathcal{D}$ such that ${\rm diag}\, (p',q') \ge {\rm diag}\, (p,q)$ and ${\rm diag}\, (p',q') \not= {\rm diag}\, (p,q)$. In order to prove that ${\rm diag}\, (p,q)$ is a maximal element of $\mathcal{D}_A$ we need to verify that \begin{equation}\label{mmm}{\rm diag}\, (p',q') \not\le A.\end{equation} If $p' > t$ or $q' > s$ then the verification of (\ref{mmm}) is trivial. So, we may assume that $p' \le t$ and $q' \le s$. We have $p' >p$ or $q' > q$ and therefore $(t-p')(s-q') < (t-p)(s-q) = u^2$ yielding that the determinant of $A - {\rm diag}\, (p',q')$ is negative, and consequently, (\ref{mmm}) holds in this case, as well. Hence, every element of $\mathcal{S}$ is a maximal element of $\mathcal{D}_A$.

Consider now ${\rm diag}\, (p,q) \in \mathcal{D}_A$ such that  ${\rm diag}\, (p,q) \not\in \mathcal{S}$. Then $(t-p)(s-q) > u^2$ yielding that $p <t$ and $q<s$. Then we can find $p'$, $p < p' < t$ such that 
$(t-p')(s-q) > u^2$ which implies that ${\rm diag}\, (p',q) \in \mathcal{D}_A$. Thus, ${\rm diag}\, (p,q)$ is not a maximal element of  $\mathcal{D}_A$. The claim that the set of maximal elements of $\mathcal{D}_A$  coincides with $\mathcal{S}$ has been proved.

We are now ready to prove our equivalence. Assume first that $A$ and $B$ satisfy the second condition. It is clear that if $A=B$ then 
$\mathcal{D}_A = \mathcal{D}_B$. If $A,B \in \mathcal{R} \cup \{ 0 \}$ then 
$\mathcal{D}_A = \mathcal{D}_B = \{ 0 \}$. So, assume that $A= JBJ$. Then obviously $B = JAJ$. If $X \in \mathcal{D}_A$ then $X \le A$ and therefore $X = JXJ  \le JAJ = B$. The proof in one direction is completed.

Hence, assume that $\mathcal{D}_A = \mathcal{D}_B$. We start with the possibility that  $\mathcal{D}_A = \mathcal{D}_B = \{ 0 \}$. Then we know that 
$A,B \in \mathcal{R} \cup \{ 0 \}$. Next, if the set of all maximal elements of $\mathcal{D}_A$ coincides with the set of all maximal elements of $\mathcal{D}_B$ and this set is a singleton $\{ C \}$, $C\not= 0$, then we have already shown that $A=B=C$ is a diagonal matrix. It remains to consider the case when the set of all maximal elements of
$\mathcal{D}_A$, which is equal to  the set of all maximal elements of
$\mathcal{D}_B$, is an infinite set. Then both $A$ and $B$ are invertible matrices with nonzero off-diagonal entries,
$$
A =  \left[ \begin{matrix} t& u \cr u & s \cr \end{matrix} \right] \ \ \ {\rm and} \ \ \ B =  \left[ \begin{matrix} t'& u' \cr u' & s' \cr \end{matrix} \right].
$$
The set of maximal elements of $\mathcal{D}_A$ is the same as the set 
of maximal elements of $\mathcal{D}_B$, that is, the two sets below
$$
\{ (p,q)\, : \, 0\le p \le t \ \, {\rm and} \ \, 0 \le q \le s \ \, {\rm and} \ \, (t-p)(s-q) = u^2  \}$$ and $$
\{ (p,q)\, : \, 0\le p \le t' \ \, {\rm and} \ \, 0 \le q \le s' \ \, {\rm and} \ \, (t'-p)(s'-q) = u'^2  \} 
$$
are the same. From here one can easily conclude that $t=t'$, $s=s'$ and $u = \pm u'$. It follows that $A=B$ or $A = JBJ$.
\end{proof}

\begin{lemma}\label{twocompare}
Let $O$ and $L$ be $2\times 2$ orthogonal matrices. Assume that 
for every pair $X,Y \in \mathcal{D}$ we have
\begin{equation}\label{jeb1}
X\le OYO^t \iff  X \le LYL^t 
\end{equation}
and
\begin{equation}\label{jeb2}
X^\sharp \le OYO^t \iff  X^\sharp \le LYL^t .
\end{equation}
Then $LYL^t = OYO^t$ for every $Y \in \mathcal{D}$.
\end{lemma}

\begin{proof}
The assumption (\ref{jeb1}) can be rewritten as $\mathcal{D}_{OYO^t} = \mathcal{D}_{LYL^t}$, $Y \in \mathcal{D}$. By Lemma \ref{max} we see that for every $Y \in \mathcal{D}$ of rank two we have either
\begin{equation}\label{mrk1}
OYO^t = LYL^t , \ \ \ {\rm or} \ \ \ OYO^t = JLYL^t J,
\end{equation}
where 
$$
J= \left[ \begin{matrix} 1& 0 \cr 0 & -1 \cr \end{matrix} \right] .
$$ 
Set
$$
J^\sharp =   \left[ \begin{matrix} 1/2 & 1/2 \cr 1/2 & 1/2 \cr \end{matrix} \right] -  \left[ \begin{matrix} 1/2 & -1/2 \cr -1/2 & 1/2 \cr \end{matrix} \right]
$$
and
$$
\mathcal{D}^\sharp = \left \{   s \left[ \begin{matrix} 1/2 & 1/2 \cr 1/2 & 1/2 \cr \end{matrix} \right] +t  \left[ \begin{matrix} 1/2 & -1/2 \cr -1/2 & 1/2 \cr \end{matrix} \right] \, : \, 0 \le s,t \le 1 \right\} .
$$
Applying Lemma \ref{max} once more, this time with $\mathcal{D}^\sharp$ instead of $\mathcal{D}$ and (\ref{jeb2}) instead of (\ref{jeb1}), we conclude 
that for every $Y \in \mathcal{D}$ of rank two we have either
\begin{equation}\label{mrk2}
OYO^t = LYL^t , \ \ \ {\rm or} \ \ \ OYO^t = J^\sharp LYL^t J^\sharp .
\end{equation}
Assume first that $L$ is a diagonal orthogonal matrix. Then $J LYL^t J =  LYL^t$ and by (\ref{mrk1}) we have $OYO^t = LYL^t$ for every $Y \in \mathcal{D}$ of rank two. Since the set of rank two diagonal effects is dense in the set of all diagonal effects we have $OYO^t = LYL^t$ for all $Y \in \mathcal{D}$, as desired. 

Hence, from now on we can assume that $L \not\in \mathcal{D}$. If we have $OYO^t \not= LYL^t$ for some $Y \in \mathcal{D}$ with two different nonzero eigenvalues, then (\ref{mrk1}) and (\ref{mrk2}) imply that
$ JLYL^t J =  J^\sharp LYL^t J^\sharp$, or equivalently, 
$$
LYL^t (J  J^\sharp)= (J J^\sharp) LYL^t .
$$
Hence, $LYL^t$ commutes with
$$
JJ^\sharp =  \left[ \begin{matrix} 0& 1 \cr -1 & 0 \cr \end{matrix} \right] .
$$
A straighforward computation shows that then $LYL^t = sI$ for some real number $s$ contradicting our assumption that $Y$ has two different eigenvalues.

Hence, we have $OYO^t = LYL^t$ for all $Y \in \mathcal{D}$ with two different nonzero eigenvalues, and consequently, 
$OYO^t = LYL^t$ for all $Y \in \mathcal{D}$.
\end{proof}

The proof of the following lemma is easy and is left to the reader.

\begin{lemma}\label{dontr}
Let $s$ be a real number, $1/2 \le s \le 1$, and
$$
A =  \left[ \begin{matrix} s& 0 \cr 0 & 1 \cr \end{matrix} \right] .
$$
Assume that a $2\times 2$ real matrix $P$ is a projection of rank one with the property that for every real number $p > 1/2$ we have $pP \not\le A$. Then $s=1/2$ and
$$
P =  \left[ \begin{matrix} 1& 0 \cr 0 & 0 \cr \end{matrix} \right] .
$$
\end{lemma}

\section{Order automorphisms of $[0,I]$}

The symbol $\F$ will stand for either the real field, or the complex field. Let $H$ be a real or complex Hilbert space. A map $T:H \to H$ is called a bounded bijective semilinear operator if it is bijective, continuous, and linear in the real case, and linear or conjugate-linear in the complex case. We should remark that this does not conflict the usual definition of semilinear maps. Indeed, we call a map $T : H \to H$ semilinear if there exists an automorphism $f: \F \to \F$ such that for every $x,y \in H$ and $\lambda \in \F$ we have $T(x+y) = Tx + Ty$ and $T(\lambda x ) = f(\lambda) Tx$. It is well-known that in the real case a map $T$ is semilinear if and only if it is linear. There are exactly two continuous automorphisms of the complex field, that are $f(\lambda) = \lambda$, $\lambda \in \CC$, and $f(\lambda) = \overline{\lambda}$, $\lambda \in \CC$, but there exist many automorphisms of the complex field that are neither the identity nor the complex conjugation. Clearly, if $H$ is a complex Hilbert space and a nonzero map $T : H \to H$ is semilinear with respect to the field automorphism $f$ and $T$ is continuous, then $f$ must be continuous, too.

Let us recall that if $H$ is a complex Hilbert space and $T : H \to H$ a bounded conjugate-linear operator then $T^\ast : H \to H$ is the unique conjugate-linear bounded operator such that $\langle Tx,y \rangle = \overline{ \langle x, T^\ast y \rangle }$, $x,y \in H$. Clearly, if $T : H \to H$ is a bijective bounded conjugate-linear operator and $A \in S(H)$ positive then $T^\ast AT$ is a bounded linear positive  operator. In particular,
$T^\ast T$ is a linear operator and $T^\ast T > 0$. Further, if $H$ is a real or complex Hilbert space and $A,B \in B(H)$, then we know that $AB+I$ is invertible if and only if $BA + I$ is invertible. Indeed, one can easily verify that if $AB+I$ is invertible and $C$ is its inverse, then $I - BCA$ is the inverse of $BA + I$.

\begin{theorem}\label{reacom} 
Let $H$ be a real or complex Hilbert space with $\dim H \ge 3$. Then the map $\phi : [0, I] \to [0, I]$ is an order automorphism if and only if
there exists a bijective semilinear bounded operator
$T: H\to H$ such that
$$
\phi (X) =  \phi_T(X) = T \left( X (T^\ast T -I) +I \right)^{-1} X T^\ast
$$
for every
$X \in [0,I]$.
If $T, S : H \to H$ are invertible bounded semilinear operators and $\phi_T (X) = \phi_S (X)$, $X \in [0,I]$, then $T = cS$ for some $c \in \F$ with $|c | =1$.
\end{theorem}

\begin{proof}
In the first part of the proof we will not use the assumption that $\dim H \ge 3$.
We start by assuming that $T : H \to H$ is a bijective semilinear bounded operator. We first need to show that $X (T^\ast T -I) +I$ is invertible for every $X \in [0,I]$. 
Actually we will prove a little bit more. Since $T^\ast T > 0$ we can find a real number $\delta \in (0,1)$ such that $T^\ast T > \delta I$. Then $\varepsilon = {\delta \over 1 - \delta} > 0$ and we will show 
 that $X (T^\ast T -I) +I$ is invertible for every $X \in [0,(1+\varepsilon)I) = \{ Z \in S(H) \, : \, 0 \le Z < (1+\varepsilon)I \}$.
In what follows we will use the observations formulated in the last paragraph before the formulation of our theorem. So, let $X \in [0,(1+\varepsilon) I)$. We have $T^\ast T - I  > - (1- \delta)I$ which clearly yields
$$
X^{1/2} (  T^\ast T - I ) X^{1/2} \ge  -(1-  \delta) X. 
$$
Because $1 - \delta >0$ and $-X > -(1 + \varepsilon)I$ we further see that
$$
X^{1/2} (  T^\ast T - I ) X^{1/2} > -(1-  \delta) (1 + \varepsilon) I = -I. 
$$
This yields that $X^{1/2} (  T^\ast T - I ) X^{1/2} + I$ is invertible which further implies that $X (T^\ast T -I) +I$ is invertible, as desired.

In order to see that $\phi_T$ is well-defined we also need to verify that  for every $X \in [0 ,I]$ the operator $(X (  T^\ast T - I ) + I)^{-1} X$ is self-adjoint. The short computation below is copied from \cite{MS}. We have
 $$
 (X (  T^\ast T - I )+I)^{-1}X $$ $$
= (X (  T^\ast T - I )+I)^{-1}X( (  T^\ast T - I )X+I)( (  T^\ast T - I )X+I)^{-1}$$
$$= (X (  T^\ast T - I )+I)^{-1}(X (  T^\ast T - I )+I)X( (  T^\ast T - I )X+I)^{-1} $$ $$
=X( (  T^\ast T - I )X+I)^{-1},
$$
as desired.

So far we have shown that $\phi_T$ is a well-defined map from $[0, (1 + \varepsilon) I)$ to $S(H)$. Let us next prove that for any pair $X,Y \in [0,I]$ we have $X \le Y \Rightarrow \phi_T (X) \le \phi_T (Y)$. Let us start with the case that $0 < X \le Y < (1+\varepsilon) I$. It is well-known that from $X \le Y$ it follows that $Y^{-1} \le X^{-1}$ which further yields $T^\ast T - I + Y^{-1} \le T^\ast T - I + X^{-1}$. We have
$$
T^\ast T - I + Y^{-1} > \delta I - I + {1 \over 1+ \varepsilon} I = 0, 
$$
and consequently,   
$$
\left(  T^\ast T - I + X^{-1} \right)^{-1} \le \left(  T^\ast T - I + Y^{-1} \right)^{-1}.
$$
Thus, $ (X (  T^\ast T - I )+I)^{-1}X \le  (Y (  T^\ast T - I )+I)^{-1}Y$ which finally implies $\phi_T(X) \le \phi_T (Y)$. Hence, if $X,Y \in [0,I]$ and $X  \le Y$ then for any positive integer $n$ large enough we have 
$0 < X + (1/n) I \le Y + (1/n) I < (1+\varepsilon) I$ and therefore, $\phi_T(X + (1/n)I) \le \phi_T (Y + (1/n)I)$. By the continuity we have $\phi_T(X) \le \phi_T (Y)$. Since $\phi_T (0) = 0$ and $\phi_T (I) = I$ we conlude that $\phi_T$ maps the effect algebra $[0,I]$ to itself and it preserves order in one direction. It remains to show that $\phi_T : [0,I] \to [0,I]$ is bijective and that $X,Y \in [0,I]$ and $X \le Y$
yield $\phi_T^{-1} (X) \le \phi_T^{-1} (Y)$.

Let $S,R : H \to H$ be bijective bounded semilinear operators. For any $X \in [0,I]$ we have
$$
\phi_S ( \phi_R (X)) = $$ $$S \left[ R ( X(R^\ast R -I) + I)^{-1} XR^\ast (S^\ast S - I) + I \right]^{-1} R [ X(R^\ast R -I) + I]^{-1} XR^\ast S^\ast
$$
$$
=S \left(       ( X(R^\ast R -I) + I) R^{-1} \left[     R ( X(R^\ast R -I) + I)^{-1} XR^\ast (S^\ast S - I) + I \right]     \right)^{-1} XR^\ast S^\ast
$$
$$
= S \left( XR^\ast (S^\ast S - I) + XR^\ast - XR^{-1} + R^{-1} \right)^{-1} XR^\ast S^\ast
 $$
$$
= SR (XR^\ast S^\ast SR - X + I)^{-1} XR^\ast S^\ast = (SR)\left( X ( (SR)^\ast (SR) -I) + I \right)^{-1} X (SR)^\ast
$$
$$
= \phi_{SR} (X).
$$
Because $\phi_I (X) = X$, $X \in [0,I]$, we conclude that $\phi_T : [0,I] \to [0,I]$ is bijective and $\phi_T^{-1} = \phi_{T^{-1}}$. We already know that 
$ \phi_{T^{-1}}$ preserves order. This completes the proof of one direction.

We have
$$
\phi_T ( (1/2)I) = T (I+T^\ast T)^{-1} T^\ast = \left( (T^\ast)^{-1} (I + T^\ast T) T^{-1} \right)^{-1} = (I + S)^{-1},
$$
where $S = (TT^\ast)^{-1}$ runs over all positive invertible linear  operators as $T$ runs over all semilinear bounded invertible operators on $H$. Hence, $\phi_T ( (1/2)I) \in (0,I)$ and for every $C \in (0,I)$ we can find an invertible linear bounded operator $T: H \to H$ such that $\phi_T ( (1/2)I) = C$. It follows easily that for every semilinear bounded invertible operator $T : H \to H$ we have $\phi_T ((0,I)) = (0,I)$.

Assume now that $\phi : [0,I] \to [0,I]$ is an order automorphism.
We already know that there is no loss of generality in assuming that $\phi ((1/2)I) = (1/2)I$. 

By Lemma \ref{rkone}, an operator $A \in [0,I]$ is of rank at most one if and only if
for every pair $C,D \in [0,I]$ satisfying $C,D \le A$ we have $C\le D$ or $D \le C$. This characterization of effects of rank at most one and $\phi(0) = 0$ imply that $\phi$ maps the set of rank one effects onto itself.

Clearly, the set of projections of rank one is the set of maximal elements in the set of all rank one effects with respect to the order $\le$.

Thus, if we denote by $\mathcal{P}_1 \subset [0,I]$ the set of all projections of rank one, then there exists a bijective map $\varphi : \mathcal{P}_1 \to \mathcal{P}_1$ such that $\phi (P) = \varphi (P)$ for every $P \in \mathcal{P}_1$. It follows that the set of effects $A$ satisfying $A \le P$ is mapped onto the set of 
effects $B$ satisfying $B \le \phi(P) = \varphi (P)$, or equivalently, for every $P \in \mathcal{P}_1$ there exists a 
bijective increasing function $f_P : [0,1] \to [0,1]$ such that $\phi (tP) = f_P (t) \varphi (P)$, $0 \le t \le 1$. 
Further, for any real $t$, $0 \le t \le 1$, and any $P \in \mathcal{P}_1$ we have $tP \le (1/2)I$ if and only if $f_P (t) \varphi(P) \le \phi ((1/2)I)= (1/2)I$, which gives us $f_P(1/2) = 1/2$, that is, 
for every $P \in \mathcal{P}_1$ we have $\phi ((1/2)P) = (1/2)\varphi (P)$. 

In the next step we will see that the set of projections of rank two is mapped by $\phi$ bijectively onto the set of projections of rank two and if $R$ is a projection of rank two then $\phi ((1/2)R) = (1/2)\phi (R)$. Take any projection of rank two. It can be written as $P+Q$ where $P,Q$ is an orthogonal pair of projections of rank one. Then $\varphi (P), \varphi (Q) \le \phi (P+Q) \le I$, and clearly, $\varphi(P) \not= \varphi (Q)$.
By Lemma \ref{projtwo} we have
$$
\phi (P+Q) = \left[ \begin{matrix} I_K & 0 \cr 0 & B \cr \end{matrix} \right],
$$
where $I_K$ stands for the identity operator on 
$K = {\rm Im}\, \varphi (P) \oplus {\rm Im}\, \varphi (Q)$ 
(note that at the moment we do not know whether this direct sum is orthogonal), 
and $B :  K^\perp \to K^\perp$ is a positive operator with 
$B \le I_{K^\perp}$. Using the same arguments for the order automorphism $\phi^{-1}$ we see that
$$
\phi^{-1} \left(  \left[ \begin{matrix} I_K & 0 \cr 0 & 0 \cr \end{matrix} \right] \right) =  \left[ \begin{matrix} I_L & 0 \cr 0 & R \cr \end{matrix} \right],
$$
where $L =  {\rm Im}\, P \oplus {\rm Im}\, Q$ and $R$ is a positive operator. Furthermore,
$$
\left[ \begin{matrix} I_L & 0 \cr 0 & R \cr \end{matrix} \right] = \phi^{-1} \left(  \left[ \begin{matrix} I_K & 0 \cr 0 & 0 \cr \end{matrix} \right] \right)  \le \phi^{-1} \left(  \left[ \begin{matrix} I_K & 0 \cr 0 & B \cr \end{matrix} \right] \right) 
= \left[ \begin{matrix} I_L & 0 \cr 0 & 0 \cr \end{matrix} \right],
$$
which yields that $R=0$ and
$$
\phi^{-1} \left(  \left[ \begin{matrix} I_K & 0 \cr 0 & 0 \cr \end{matrix} \right] \right)  = \phi^{-1} \left(  \left[ \begin{matrix} I_K & 0 \cr 0 & B \cr \end{matrix} \right] \right) .
$$
It follows that $B=0$. Thus, $\phi (P+Q)$ is a projection of rank two.

Set $R = P+Q$. Then $(1/2)P , (1/2)Q \le (1/2)R \le (1/2)I$, and consequently, $\phi((1/2)R) \le \phi (R)$ is of rank two and satisfies
$$
(1/2) \varphi (P) , (1 /2) \varphi (Q) \le \phi ((1/2)R) \le (1/2)I .
$$
Using a slight modification of Lemma \ref{projtwo} we see that $ \phi ((1/2)R)$ is a rank two projection onto the two-dimensional subspace spanned by the images of $\varphi (P)$ and $\varphi (Q)$ multiplied by $1/2$. It follows that 
$\phi ((1/2)R) = (1/2)\phi (R)$.

We consider $P,Q \in \mathcal{P}_1$ that are orthogonal and denote $A = (1/2) P + Q$. Then $(1/2) \phi (P+Q) = \phi ((1/2) (P+Q)) \le \phi (A) \le \phi (P+Q)$. Therefore,
$$
\phi (A) = sR + tR' ,
$$
where $1/2 \le s \le t \le 1$, and $R,R'$ is a pair of orthogonal rank one projections such that the image of $R+R'$ is the two-dimensional subspace spanned by the images  of $\varphi (P)$ and $\varphi (Q)$. Note that the possibility that $s=t=1$ cannot occur because we know that $\phi (A)$ is not a projection of rank two.
We further know that $\varphi (Q) \le \phi (A) \le I$ which yields that $t=1$ and $R' = \varphi (Q)$.
It follows from
$p\varphi (P) \not\le \phi (A)$ whenever $p > 1/2$ and Lemma \ref{dontr} that
$\phi(A) = (1/2)\varphi (P) + \varphi (Q)$ and $\varphi (P) \perp \varphi (Q)$.

So far we have not used the assumption that $\dim H \ge 3$ and everything above holds in the case when $\dim H =2$ as well. But in the next step the assumption that $\dim H \ge 3$ is essential.
Uhlhorn's theorem, which is a straightforward consequence of the fundamental theorem of projective geometry, see for example \cite{Fa}), states that if $\varphi : \mathcal{P}_1 \to \mathcal{P}_1$ is a bijective map such that for every pair $P,Q \in \mathcal{P}_1$ we have $P \perp Q \iff \varphi(P) \perp \varphi (Q)$, then in the complex case there exists a unitary or antiunitary operator $U : H \to H$ such that $\varphi (P) = UPU^\ast$, $P \in \mathcal{P}_1$, while in the real case there exists an orthogonal operator $O : H \to H$ such that $\varphi (P) = OPO^\ast$, $P \in \mathcal{P}_1$. 
Thus, we can assume with no loss of generality that $\phi (P) = P$ for every $P \in \mathcal{P}_1$ and 
$\phi ((1/2) P + Q) = (1/2)P+Q$ for every orthogonal pair $P,Q \in \mathcal{P}_1$. 

We claim that 
\begin{equation}\label{oppos}
\phi (I-P) = I-P
\end{equation}
 for every $P \in \mathcal{P}_1$.
Indeed, for every rank one projection $Q$ that is orthogonal to $P$ we have $Q \le I -P$ and therefore $Q \le \phi (I-P)$. This yields $I-P \le \phi (I-P)$ which further implies that $\phi (I-P) = (I-P) + tP$ for some $t \in [0,1]$. For every real $s \in (0,1]$ we have $sP \not\le I-P$ and consequently, for every real $p \in (0,1]$ we have $pP \not\le \phi (I-P)$. This shows that $t=0$, as desired.

Using Lemma \ref{mikmik} we see
that $\phi (tR) = tR$ for every $R \in \mathcal{P}_1$ and every real $t \in [1/2,1]$. 

We introduce a new map $\tau : [0,I] \to [0,I]$ defined by $\tau (X) = I - \phi (I-X)$, $X \in [0,I]$. This is again an order automorphism of $[0,I]$ satisfying $\tau((1/2)I) = (1/2)I$. It follows from (\ref{oppos}) that 
$\tau (P) = P$ for every $P \in \mathcal{P}_1$. But then we know that $\tau (tR) = tR$ for every $R \in \mathcal{P}_1$ and every real $t \in [1/2,1]$. This further implies that $\phi (I -tR) = I -tR$ 
 for every $R \in \mathcal{P}_1$ and every real $t \in [1/2,1]$. 

Let $s \in [0, 1/2]$ and let $R$ be any projection of rank one. Then $t = 1-s \in [1/2 , 1]$. For every real $p \in [0,1]$ we have $p \le s$ if and only if $pR \le I -tR$ which is further equivalent to $f_R (p) R \le I - tR$. Hence, for 
every real $p \in [0,1]$ we have $p \le s$ if and only if $f_R (p) \le 1-t =s$. Thus, $f_R (s) = s$.

We have shown that $\phi (tR) = tR$ for every $R \in \mathcal{P}_1$ and every real $t \in [0,1]$. Using Lemma \ref{less} 
we immediately get that $\phi (X) = X$ for every $X \in [0,I]$.

Finally, suppose that $T, S : H \to H$ are invertible bounded semilinear operators and $\phi_T (X) = \phi_S (X)$, $X \in [0,I]$. We need to show that there exists $c \in \F$ such that $|c| = 1$ and $T=cS$. 
(While proving this fact we will also explain the formula (\ref{nice}).
At first glance this formula that describes the general form of order automorphisms of effect algebras looks quite surprising.  A careful reader will notice that it is related to the fact that the group of order automorphisms of the set of all positive invertible operators coincides with the set of all congruences $X \mapsto TXT^\ast$ where $T$ is an invertible bounded semilinear operator on $H$.) 
We already know $\phi_T ((0,I)) = (0,I)$ and $\phi_S ((0,I)) = (0,I)$. We will use the same symbols $\phi_T$ and $\phi_S$ to denote the restrictions of $\phi_T$ and $\phi_S$ to $(0,I)$, respectively. 
Obviously, the map $\psi : (0, \infty) \to (0,I)$ defined by $\psi (X) = (I+X)^{-1}$, $X \in (0, \infty)$, is an order anti-isomorphism, that is, $\psi$ is bijective and for every pair $X,Y \in (0, \infty)$ we have $X \le Y \iff \psi (Y) \le \psi (X)$.
Its inverse $\psi^{-1} : (0,I) \to (0, \infty)$ is given by $\psi^{-1} (X) = X^{-1} - I$, $X \in (0,I)$. Set $T' = (T^\ast)^{-1}$ and $S' = (S^\ast)^{-1}$ and define order automorphisms $\xi_{T'} , \xi_{S'} : (0, \infty) \to (0, \infty)$ by $\xi_{T'} (X) = T'XT'^\ast$ and $\xi_{S'} (X) = S'XS'^\ast$, $X \in (0,\infty)$. For every $X \in (0,I)$ we have
$$
(\psi \circ \xi_{T'} \circ \psi^{-1}) (X)= \psi \left( (T^\ast)^{-1} (X^{-1} - I) T^{-1} \right) = \left( I +  (T^\ast)^{-1} (X^{-1} - I) T^{-1} \right)^{-1}
$$
$$
= \left(  (T^\ast)^{-1} \left( T^\ast T +  (X^{-1} - I)\right)  T^{-1} \right)^{-1} = T \left( T^\ast T - I + X^{-1}\right)^{-1} T^\ast
$$
$$
= T  \left( X^{-1} \left( X (T^\ast T - I) + I \right) \right)^{-1} T^\ast = T  \left( X (T^\ast T - I) + I  \right)^{-1} XT^\ast = \phi_T (X).
$$
Therefore $\phi_T (X) = \phi_S (X)$, $X \in (0,I)$, yields that $\xi_{T'} (X) = \xi_{S'} (X)$, $X \in (0, \infty)$, or equivalently,
$$
T' X T'^\ast = S' X S'^\ast 
$$
for every $X \in (0, \infty)$. It is very easy to conlude that $T' = cS'$ for some $c \in \F$ with $|c| = 1$. The desired equality $T = cS$ is an easy consequence.
\end{proof}

\section{Two-dimensional case}

In this section we will deal with the two-dimensional case. In the finite-dimen\-sio\-nal case we usually identify operators with matrices. Therefore we will formulate the two theorems below in the language of matrices. 

\begin{theorem}\label{reatwo} 
Let $0$ and $I$ denote the $2 \times 2$ zero matrix and the $2 \times 2$ identity matrix, respectively, and let $[0,I]$ be the matrix interval in $S_2$. Then the map $\phi : [0, I] \to [0, I]$ is an order automorphism if and only if
there exists an invertible $2\times 2$ real matrix $T$
such that
$$
\phi (X) =  T \left( X (T^t T -I) +I \right)^{-1} X T^t
$$
for every
$X \in [0,I]$.
\end{theorem}

\begin{theorem}\label{comtwo} 
Let $0$ and $I$ denote the $2 \times 2$ zero matrix and the $2 \times 2$ identity matrix, respectively, and let $[0,I]$ be the matrix interval in $H_2$,  the set of all $2\times 2$ complex hermitian matrices. Then the map $\phi : [0, I] \to [0, I]$ is an order automorphism if and only if
there exists an invertible $2\times 2$ complex matrix $T$
such that either
$$
\phi (X) =   T \left( X (T^\ast T -I) +I \right)^{-1} X T^\ast
$$
for every
$X \in [0,I]$; or
$$
\phi (X) =  T \left( X^{t}  (T^\ast T -I) +I \right)^{-1} X^{t} T^\ast
$$
for every
$X \in [0,I]$.
\end{theorem}

Note that for every complex hermitian matrix $X$ we have $X^{t} = \overline{X}$, where $\overline{X}$ is the matrix obtained from $X$ by aplying the complex conjugation entrywise. This explains the appearance of $X^{t}$ in one of the formulae describing order automorphisms of effect algebras in the finite-dimensional case when we use the matrix language rather than speaking of linear or conjugate-linear operators.

Note that the above two theorems are generalizations of the main theorem in \cite{MoP}.

Again we could have proved both theorems simultaneously. But since Theorem \ref{comtwo} is already known, see \cite{MS}, and the proof of Theorem \ref{comtwo} can be easily obtained by just slightly modifying the proof of Theorem \ref{reatwo}, we will give just the proof of the real case.

\begin{proof}[Proof of Theorem \ref{reatwo}]
We need to prove just one direction. So, assume that $\phi : [0, I] \to [0,I]$ is an order automorphism. We already know that there is no loss of generality in assuming that $\phi ((1/2)I) = (1/2)I$. Once we restrict ourseleves to this special case we further know that there exists a bijective map $\varphi : \mathcal{P}_1 \to \mathcal{P}_1$ such that $\phi (P) = \varphi (P)$ and $\phi ((1/2)P) = (1/2)\varphi (P)$, $P \in \mathcal{P}_1$, and for every pair of orthogonal rank one projections $P,Q \in \mathcal{P}_1$ we have $\varphi (P) \perp
\varphi (Q)$ and $\phi ((1/2)P+Q) = (1/2) \varphi (P) + \varphi (Q)$. Moreover,  for every $P \in \mathcal{P}_1$ there exists a 
bijective increasing function $f_P : [0,1] \to [0,1]$ such that $\phi (tP) = f_P (t) \varphi (P)$, $0 \le t \le 1$.

Our next goal is to prove that for every rank one projection $Q$ we have $\phi ((1/3)Q) = (1/3)\varphi (Q)$. Set $A = (1/3)Q + (I-Q)$ and find a rank one projection $P$ such that ${\rm tr}\, (PQ) = 1/2$. Then by Lemma  \ref{chrono} 
there exist nonnegative real numbers $t_1 , t_2 , t_3$ and rank one projections $R_1 , R_2 , R_3$ such that
$$
A = (1/2)P + t_1 R_1 , \ \ \ A = (1/2)(I-P) + t_2 R_2 , \ \ \ {\rm and} \ \ \ A = I- t_3 R_3.  
$$ 
Using Lemma \ref{rkone} three times together with the fact that $\phi : [0, I] \to [0,I]$ is an order automorphism we see that there exist nonnegative real numbers $s_1 , s_2 , s_3$ and rank one projections $R'_1 , R'_2 , R'_3$ such that
$$
\phi (A) = (1/2)\varphi(P) + s_1 R'_1 , \ \ \ \phi (A) = (1/2)(I-\varphi(P)) + s_2 R'_2 , 
$$
and
$$ 
\phi(A) = I- s_3 R'_3.  
$$ 
Applying Lemma \ref{chrono} once more we conclude that
$$
\phi (A) = \phi (  (1/3)Q + (I-Q) ) =  (1/3)Q' + (I-Q')
$$
for some rank one projection $Q'$ satisfying
$$
{\rm tr}\, (\varphi (P) Q') = 1/2.
$$
Because $I-Q \le A$ we have
$$
I - \varphi(Q) \le \phi (A) =  (1/3)Q' + (I-Q').
$$
It follows that $I-Q' = I - \varphi (Q)$, or equivalently, $Q' = \varphi (Q)$. 

We know that there exists a bijective increasing function $k : [0,1] \to [0,1]$ such that $\phi (tQ) = k(t) Q'$. Further, we have
$$
t \le 1/3 \iff tQ \le A \iff k(t) Q' \le (1/3)Q' + (I-Q'),
$$
and consequently, $k((1/3)) = 1/3$, that is, $\phi ((1/3)Q) = (1/3)\varphi (Q)$. 

Because $Q$ is an arbitrary projection of rank one it follows trivially that $\phi ((1/3)I) = (1/3)I$.

Note that we have also shown the following: If $P,Q$ are projections of rank one then  \begin{equation}\label{trhalf} {\rm tr} \, (PQ) = 1/2 \Rightarrow {\rm tr} \, (\varphi(P) \varphi(Q)) = 1/2. \end{equation}

As in the previous section we introduce a new map $\tau : [0,I] \to [0,I]$ defined by $\tau (X) = I - \phi (I-X)$, $X \in [0,I]$. Clearly, $\tau ((1/2)I) = (1/2)I$. By the previous step of the proof we conclude that 
$\tau ((1/3)I) = (1/3)I$, or equivalently, $\phi ((2/3)I) = (2/3)I$. From here we immediately get that  for every rank one projection $Q$ we have $\phi ((2/3)Q) = (2/3)\varphi (Q)$.

Let us summarize what we have obtained so far. We have assumed that $\phi : [0,I] \to [0, I]$ is an order automorphism and $\phi (tI) = tI$ for $t= 0, 1/2, 1$.
It follwows that there exists a bijective map $\varphi : \mathcal{P}_1 \to \mathcal{P}_1$ such that $\varphi (I-P) = I - \varphi (P)$, $P \in \mathcal{P}_1$, and
$$
\phi (tI) = tI \ \ \ {\rm and} \ \ \ \phi (tP) = t \varphi (P)
$$
for every $t=0, 1/3, 1/2, 2/3, 1$ and every $P \in \mathcal{P}_1$.

This further yields that for any pair of real numbers $a,b$, $0 \le a < b \le 1$, the following holds true: 
Assume that $\phi (tI) = tI$ for $t= a, (1/2) (a+b), b$.
Then 
$$
\phi (tI) = tI 
$$
for every $t=a, (2/3)a + (1/3)b, (1/2) (a+b), (1/3)a + (2/3)b, b$, and consequently,
$$
\phi (tP) = t \varphi(P)
$$
for every $t=a, (2/3)a + (1/3)b, (1/2) (a+b), (1/3)a + (2/3)b, b$, and
every $P \in \mathcal{P}_1$.

We apply the above statement to matrix intervals $[0, (2/3)I]$, $[(1/3)I, (2/3)I]$, and $[(1/3)I, I]$ to conclude that
$$
\phi (tI) = tI \ \ \ {\rm and} \ \ \ \phi (tP) = t \varphi (P)
$$
for every $t=0, 2/9, 3/9, 4/9, 5/9, 6/9, 7/9, 1$ and every $P \in \mathcal{P}_1$. 

We have shown that if $\phi : [0,I] \to [0, I]$ is an order automorphism and $\phi (tI) = tI$ for $t= 0, 1/2, 1$, then for $t_1 = 0$,  $t_2 = {2\over 9}$,
 $t_{3} = {4\over 9}$,  $t_{4} = {3 \over 9}$,  $t_{5} = {5\over 9}$,
 $t_{6} = {7\over 9}$,  $t_{7} = {5\over 9}$,  $t_{8} = {7\over 9}$, and
 $t_{9} = 1$ we have
$$
[0,1] = [t_1 , t_3 ] \cup  [t_4 , t_6 ] \cup  [t_7 , t_9 ] ,
$$
$$
t_{3k+3}  - t_{3k+1} = {4 \over 9} \ \ \ {\rm and} \ \ \ t_{3k+2} = { t_{3k+1}  + t_{3k+3} \over 2}, \ \ \ k=0,1,2,
$$
and 
$$
\phi (t_m I) = t_m I \ \ \ {\rm and} \ \ \ \phi (t_m P) = t_m \varphi (P), \ \ \ m = 1, \ldots, 9,
$$
for every $P \in \mathcal{P}_1$. 
 
From here one can easily see that 
for any pair of real numbers $a,b$, $0 \le a < b \le 1$, the following holds true: 
Assume that  $\phi (tI) = tI$ for $t= a, (1/2) (a+b), b$.
Then there exist real numbers $s_1, \ldots, s_9$ such that 
$$
[a,b] = [s_1 , s_3 ] \cup  [s_4 , s_6 ] \cup  [s_7 , s_9 ] ,
$$
$$
s_{3k+3}  - s_{3k+1} = {4 \over 9} (b-a) \ \ \ {\rm and} \ \ \ s_{3k+2} = { s_{3k+1}  + s_{3k+3} \over 2}, \ \ \ k=0,1,2,
$$
and 
$$
\phi (s_m I) = s_m I \ \ \ {\rm and} \ \ \ \phi (s_m P) = s_m \varphi (P), \ \ \ m = 1, \ldots, 9,
$$
for every $P \in \mathcal{P}_1$.

A simple inductive argument shows that for every positive integer $r$ there exist real numbers $t_{r,1}, \ldots , t_{r, 3^{r+1}}$ such that 
$$
[0,1] = \bigcup_{j=0}^{3^r-1} \left[ t_{r, 3j+1} , t_{r, 3j+3}  \right]  ,
$$
$$
t_{r, 3j+3}  - t_{r, 3j+1} = \left({4 \over 9}\right)^r \ \ \ {\rm and} \ \ \ t_{r, 3j+2} = { t_{r, 3j+1}  + t_{r, 3j+3} \over 2}, \ \ \ j=0,1,\ldots , 3^r -1,
$$
and 
$$
\phi (t_{r,m} I) = t_{r,m} I \ \ \ {\rm and} \ \ \ \phi (t_{r,m} P) = t_{r,m} \varphi (P), \ \ \ m = 1, \ldots, 3^{r+1}, 
$$
for every $P \in \mathcal{P}_1$.

The subset $\mathcal{T} \subset [0,1]$ given by
$$
\mathcal{T} = \bigcup_{r=1}^{\infty}      \left\{ t_{r,1}, t_{r,2}, \ldots , t_{r, 3^{r+1}} \right\}
$$
is dense in the unit interval $[0,1]$. Moreover, we have
$$
\phi (tI) = tI \ \ \ {\rm and} \ \ \ \phi (tP) = t \varphi (P)
$$
for every $t \in \mathcal{T}$and every $P \in \mathcal{P}_1$. It follows that for every  $P \in \mathcal{P}_1$ and every $t \in [0,1]$ we have 
$\phi (tI) = tI$ and $\phi (tP) = t \varphi (P)$.

Let $A \in [0,I]$ be any effect. We know that $A = tQ + s(I-Q)$ for some rank one projection $Q$ and some real numbers $s,t \in [0,1]$. We claim that $\phi (A) = t \varphi (Q) + s (I - \varphi (Q))$. We already know that this is true when $t=s$. So, assume that $t\not=s$. Without loss of generality we can assume that $t < s$. We know that for $p \in [0,1]$ we have $pI \le A$ if and only if $p \le t$. Further, $pI \ge A$ if and only if $p \ge s$.  It follows that for every $p \in [0,1]$ we have
$$
pI \le \phi (A) \iff p \le t 
$$
and
$$
pI \ge \phi (A) \iff p \ge s .
$$
From here it is easy to conclude that there exists a rank one projection $Q'$ such that $\phi (A) = tQ' + s(I-Q')$. We further use the fact that for every real number $p$ from the unit interval we have $p\varphi (Q) \le \phi (A)$ if and only if $p \le t$ to show that $\phi (A) = t \varphi (Q) + s (I - \varphi (Q))$, as desired.

After replacing the map $\phi$ by the map $X \mapsto O \phi (X) O^t$, $X \in [0,I]$, where $O$ is an appropriate orthogonal $2\times 2$ matrix, we can assume with no loss of generality that $\phi (X) = X$ for every $X \in \mathcal{D}$.
Recall that the symbol $\mathcal{D}$ stands for the set of all diagonal effects.

It follows from (\ref{trhalf}) that either
$$
\phi \left( \left[ \begin{matrix} 1/2 & 1/2 \cr 1/2 & 1/2 \cr \end{matrix} \right] \right) =  \left[ \begin{matrix} 1/2 & 1/2 \cr 1/2 & 1/2 \cr \end{matrix} \right] ,
$$
or
$$
\phi \left( \left[ \begin{matrix} 1/2 & 1/2 \cr 1/2 & 1/2 \cr \end{matrix} \right] \right) =  \left[ \begin{matrix} 1/2 & -1/2 \cr -1/2 & 1/2 \cr \end{matrix} \right] .
$$
After replacing $\phi$ by the map
$$
X \mapsto \left[ \begin{matrix} 1 & 0 \cr 0 & -1 \cr \end{matrix} \right] \phi (X) 
 \left[ \begin{matrix} 1 & 0 \cr 0 & -1 \cr \end{matrix} \right], \ \ \ X \in [0,I],
$$
if necessary,
we may assume with no loss of generality that we have the first possibility. It follows that for every $X \in \mathcal{D}$ we have $\phi (X^\sharp) = X^\sharp$, where $X^\sharp$ is defined as in (\ref{sharp}).

Moreover, we know that for every $2\times 2$ orthogonal matrix $O$ there exists a $2\times 2$ orthogonal matrix $L$ such that
$$
\phi (OXO^t) = LXL^t
$$
for every $X \in \mathcal{D}$. Choose and fix $O$ and $L$ as above. Then for every pair of diagonal effects $X,Y$ we have
$$
X\le OYO^t \iff \phi (X) = X \le LYL^t
$$
and
$$
X^\sharp\le OYO^t \iff \phi (X^\sharp) = X^\sharp \le LYL^t .
$$
By Lemma \ref{twocompare} we have $
\phi (OYO^t) = OYO^t
$
for every $Y \in \mathcal{D}$. Since $O$ can be chosen to be any $2 \times 2$ orthogonal matrix we have $\phi (X) = X$, $X \in [0,I]$. This completes the proof.  
\end{proof}

In this paper an elementary proof of the description of the group of order autoumorphisms of operator/matrix interval $[0,I]$ has been given. The only nontrivial tool that we have used is the fundamental theorem of projective geometry. If one would like to have a completely self-contained elementary proof then one could start with the elemntary self-contained proof of the two-dimensional case as presented in this section and then deduce the general case in two steps. Let $\phi : [0,I] \to [0,I]$ be an order automorphism. In the first very easy step one would verify that for every projection $P \in [0,I]$ of rank two the effect $\phi (P)$ is a projection of rank two and $\phi$ maps $\{ A \in [0,I] \, : \, PAP = A \} = P[ 0,I]P$ onto $\phi (P) [0,I] \phi (P)$ and the restriction of $\phi$ to $P [0,I]P$ is an order isomorphisms of $P[ 0,I]P$ onto $\phi (P) [0,I] \phi (P)$. This restriction can be considered as an order automorphism of the effect algebra on a two-dimensional Hilbert space and is, by the main result of this section, of a nice form. In the second step one would use this nice behaviour on each of the ``smal pieces" $P[0,I]P$,  $P \in \mathcal{P}_1$, to prove that $\phi$ behaves nicely on the whole effect algebra $[0,I]$.

\section{Order isomorphisms of matrix intervals}

Throughout this section we will assume that $n \ge 2$.
Let $A,B$ be real $n \times n$ symmetric matrices such that $A < B$. Then we define matrix intervals
$$
[A,B] = \{ C \in S_n \, : \, A \le C \le B \},
$$
$$
[A,B) = \{ C \in S_n \, : \, A \le C < B \},
$$
and
$$
(A,B) = \{ C \in S_n \, : \, A < C < B \}.
$$
Similarly, we define
$$
[A, \infty ) = \{ C \in S_n \, : \, C \ge A \},
$$
$$
(A, \infty ) = \{ C \in S_n \, : \, C > A \},
$$
and $S_n = (-\infty , \infty)$. The notations $(A,B]$, $(-\infty, A]$, and $(-\infty , A)$ should now be self-explanatory.

We will first answer the question which of the above matrix intervals are order isomorphic. 
In the next step we will describe  the general form of all order isomorphisms between matrix intervals that are order isomorphic.
It is important to note that if matrix intervals $I_1, I_2, J_1 , J_2$ are order isomorphic and the order isomorphisms $\varphi_j : I_j \to J_j$, $j=1,2$, are given, and we know
the description of the general form
of order isomorphisms between $I_1$ and $I_2$, then we immediately get the general form of order isomorphisms between $J_1$ and $J_2$.
Indeed, each order isomorphism $\psi : J_1 \to J_2$ is of the form 
$$
\psi = \varphi_2 \,  \phi \,  \varphi_{1}^{-1},
$$
where $\phi : I_1 \to I_2$ is any order isomorphism. Of course, a similar reduction is possible if we know that certain matrix intervals are order anti-isomorphic.
We have the following reduction principle: if $I_1$ and $I_2$ are order isomorphic matrix intervals, and 
$J_1$ and $J_2$ are order isomorphic matrix intervals, and $\varphi_ j : I_j \to J_j$, $j=1,2$, are order anti-isomorphisms, then 
 each order isomorphism $\psi : J_1 \to J_2$ is of the form 
$\psi = \varphi_2 \,  \phi \,  \varphi_{1}^{-1}$, where $\phi$ is an order isomorphism of $I_1$ onto $I_2$.

Let $A,B \in S_n$ with $A < B$. Then $\phi : [0,I] \to [A,B]$ given by
$$
\phi (X) = (B - A)^{1/2} X  (B - A)^{1/2} 
+ A, \ \ \ X \in [0,I],
$$
is an order isomorphism. It follows that all matrix intervals of the form $[A,B]$ with $A < B$ are isomorphic and we know how to construct an order isomorphism between two such matrix intervals. Similarly, 
all matrix intervals of the form $[A,B)$ with $A < B$ are isomorphic and again it is easy to give a simple explicit formula of  an order isomorphism between two such matrix intervals.
Analogous statements for the operator intervals $(A,B]$ and $(A,B)$ are obviously true.

Clearly, any two distinct matrix intervals from the collection $[0, I]$, $[0, I)$, $(0, I]$, and $(0,I)$ are order non-isomorphic.
Let us show that $[0, I)$ and $(0,I)$ are not order isomorphic. All we need is to observe
$0 \in [0,I)$ has the property that $0 \le A$ for every $A\in [0, I)$, while there does not exist $B \in (0, I)$
such that $B \le A$ for every $A\in (0,I)$.
  
Further, let $A \in S(H)$. Then the matrix interval $[A, \infty )$ is order isomorphic to $[0, \infty)$ via the translation isomorphism $X \mapsto X-A$.
Similarly, for every $A \in S(H)$ the matrix interval  $(-\infty , A ]$ is order isomorphic to $(- \infty, 0]$, while
clearly, $[0, \infty)$ and $(-\infty , 0]$ are not order isomorphic.

The map $\phi : [0, I) \to [0, \infty)$ given by
$$
\phi (X) = (I -X)^{-1} - I , \ \ \ X \in [0, I),
$$
is an example of an order isomorphism of $[0,I)$ onto $[0, \infty)$ (one just need to check that
that the map $X \mapsto -X$ is an order anti-isomorphism of $[0,I)$ onto $(-I,0]$ and that the map $Z \mapsto Z^{-1}$ is an order anti-isomorphism of $(0,I]$ onto $[I, \infty)$ ).
In an almost the same way we see that $(- \infty, 0]$ is order isomorphic to $(0, I]$.

Further, the map $\phi (X) = I - X^{-1}$  is an order isomorphism of $(0,I)$ onto $(-\infty , 0)$ and
the map $\phi (X) = (I-X)^{-1} - I$ is an order isomorphism of $(0,I)$ onto $(0, \infty)$.

We have shown the following.

\begin{theorem}\label{zadprv}
Every matrix interval $J$ is isomorphic to one of the following matrix intervals: $[0,I]$, $[0, \infty)$, $(-\infty, 0]$, $(0, \infty)$, and $(-\infty , \infty )$.
\end{theorem}

Let $J$ be any matrix interval. Using the above ideas one can easily construct an order isomorphism between $J$ and 
one of the above five matrix intervals.

The matrix intervals $(-\infty, \infty)$ and $(0,\infty)$ are not order isomorphic. This has been proved in the complex case in \cite{Mo3}. Exactly the same proof works also in the real case. It follows that any two of the above five matrix intervals are order non-isomorphic.
Further, $[0, \infty)$     and $(-\infty , 0]$    are obviously order anti-isomorphic.
Hence, to have a full understanding of the structure of all order isomorphisms between any two order isomorphic matrix intervals one only needs
to describe the general form of order automorphisms of the following four 
matrix intervals: $[0,I]$, $[0, \infty)$, $(0, \infty)$, and $(-\infty , \infty )$. The main theorem of our paper characterizes order automorphisms of $[0,I]$. Let us continue with $[0, \infty)$. We will show that every order automorphism $\phi : [0, \infty) \to [0, \infty)$ is a congruence transformation, that is, there exists a real invertible $n \times n$ matrix $T$ such that $\phi (X) = TXT^t$, $X \in [0, \infty)$.

In order to verify this statement we need some rather easy observations. As usual we identify vectors in $\mathbb{R}^n$ with $n \times 1$ matrices. If $x,y \in \mathbb{R}^n$ are nonzero vectors then $xy^t$ is a rank one matrix and every rank one matrix can be written in this form. For nonzero vectors $x,y, u, v \in \mathbb{R}^n$ we have
$$
(xy^t ) (uv^t) = (y^t u) xv^t = \langle y,u \rangle xv^t .
$$
It follows that for $ u, v \in \mathbb{R}^n$ satisfying $\langle u,v \rangle \not= -1$ the matrix $I + uv^t$ is invertible and
$$
(I + uv^t)^{-1} =  I - { 1 \over 1 + \langle u,v \rangle} uv^t .
$$

For every unit vector $x \in \mathbb{R}^n$ the rank one matrix $xx^t$ is a projection. Every rank one symmetric matrix is of the form $txx^t$ for some unit vector $x$ and some real number $t$. Let $\phi : [0,I] \to [0,I]$, $[0,I] \subset S_n$, be an order automorphism. Then we can find an invertible $2 \times 2$ real matrix $T$ such that $\phi (X) = T ( X (T^t T -I) + I)^{-1}XT^t$, $X \in [0,I]$. Denoting $A = A^t = T^t T -I  > -I$ we see that for every unit vector 
$x \in \mathbb{R}^n$ and every real $t$, $0 < t \le 1$, we have $\langle Ax, tx \rangle = t \langle Ax, x \rangle > t  \langle (-I)x, x \rangle = -t \ge -1$, and therefore 
$$
\phi (t xx^t) = tT( txx^t A + I)^{-1} xx^t T^t = tT( (tx)(Ax)^t  + I)^{-1} x(Tx)^t   
$$
$$
= tT \left( I - { 1 \over 1 + t\langle Ax,x \rangle} (tx)(Ax)^t \right)x (Tx)^t = t (Tx)(Tx)^t - { t^2 \langle Ax, x \rangle \over 1 + t \langle Ax, x \rangle } (Tx)(Tx)^t
$$
$$
= { t \over 1 +  t\langle Ax,x \rangle} (Tx)(Tx)^t.
$$
Obviously, $\langle Ax,x \rangle = \| Tx \|^2 - \| x \|^2 = \| Tx \|^2 - 1$. If we denote $y_x = {1 \over \| Tx \|} Tx$, then 
$$
\phi (t xx^t) = { t \| Tx \|^2 \over t(\|Tx \|^2 - 1 ) + 1 } y_x y_{x}^t.
$$
We have three possibilities. The first one is that $A = 0$. Then $T$ is an orthogonal transformation and
$$
\phi (X) = TXT^t, \ \ \ X \in [0,I].
$$
The second one is that there exists a unit vector $x$ such that $\| Tx \| > 1$. For such an $x$ there exists a unit vector $y$ such that for every real $t$, $0 \le t \le 1$, we have
$$
\phi (t xx^t) = f(t) yy^t,
$$
where $f$ is a linear rational function which is strictly increasing on the positive real half-line, $f(0)= 0$, $f(1) = 1$, and $$\lim_{t \to \infty} f(t) < \infty.$$
The third one is that there exists a unit vector $x$ such that $\| Tx \| < 1$. For such an $x$ there exists a unit vector $y$ such that for every real $t$, $0 \le t \le 1$, we have
$$
\phi (t xx^t) = f(t) yy^t,
$$
where $f$ is a linear rational function which is strictly increasing on some interval $[0, a)$ with $a>1$, $f(0)= 0$, and $\lim_{t \uparrow a} f(t) = \infty$.

For every pair $A,B \in S_n$, $A > 0$ and $B > 0$, and every order isomorphism $\phi : [0,A] \to [0,B]$ the map
$$
X \mapsto B^{-1/2}\phi ( A^{1/2} X A^{1 / 2})B^{-1/2}, \ \ \ X \in [0,I],
$$
is an automorphism of $[0,I]$ onto itself. It follows easily that either there exists an invertible $n \times n$ real matrix $S$ such that
$$
\phi (X) = SXS^t, \ \ \ X \in [0,A],
$$
(actually, we know that $S = B^{1/2}OA^{-1/2}$ for some orthogonal matrix $O$ but we do not need this fact for our proof)
or there exist unit vectors $x,y \in \mathbb{R}^n$ and a linear rational function $f$ which is strictly increasing on the positive real half-line, $f(0)= 0$ and $\lim_{t \to \infty} f(t) < \infty$ such that
$$
\phi (t xx^t) = f(t) yy^t
$$
for every real $t$ satisfying $ 0 \le txx^t \le A$, or
there exist unit vectors $x,y \in \mathbb{R}^n$ and a linear rational function $f$ which is strictly increasing on some interval $[0, a)$, $a > \alpha (A, xx^t)$, $f(0)= 0$  and $\lim_{t \uparrow a} f(t) = \infty$ such that
$$
\phi (t xx^t) = f(t) yy^t
$$
for every real $t$ satisfying $ 0 \le txx^t \le A$.

We are now ready to describe the general form of order automorphisms of $[0, \infty) \subset S_n$. So, assume that $\phi : [0, \infty) \to [0, \infty)$ is an order automorphism. Then there exists a positive integer $p$ such that $\phi (pI) = B_p >0$. For every integer $n \ge p$ we have $B_n = \phi (nI)  > 0$ and $\phi$ restricted to $[0, nI]$ is an order isomorphism of $[0, nI]$ onto $[0, B_n]$. 

We first assume that there exist unit vectors $x,y \in \mathbb{R}^n$ and a linear rational function $f$ which is strictly increasing on the positive real half-line, $f(0)= 0$ and $\lim_{t \to \infty} f(t) = m < \infty$ such that
\begin{equation}\label{makalen}
\phi (t xx^t) = f(t) yy^t
\end{equation}
for every real $t$, $0 \le t \le p$. For every positive integer $n > p$ we consider the restriction of $\phi$ to $[0, nI]$ which is an order isomorphism of $[0, nI]$ onto $[0, B_n]$. We have three possibilities for the beahviour of the mapping $t \mapsto \phi (txx^t)$, $0 \le t \le n$, and because we know that (\ref{makalen}) holds for every real $t$, $0 \le t \le p$, we conclude that 
$\phi (t xx^t) = f(t) yy^t$
for every real $t$, $0 \le t \le n$. But $n$ was an arbitrary integer greater than $p$, and thus $\phi (t xx^t) \le mI$ for every positive real $t$. Since $\phi$ is surjective we can find $A \in [0, \infty)$ such that $\phi (A) > mI$. Consequently, $\phi (A) > mI \ge \phi (t xx^t)$ for every nonnegative real $t$ implying that $txx^t \le A$ for every real $t$, $ 0 \le t < \infty$, a contradiction.

The possibility that there exist unit vectors $x,y \in \mathbb{R}^n$ and a linear rational function $f$ which is strictly increasing on some interval $[0, a)$, $a > p$, $f(0)= 0$  and $\lim_{t \uparrow a} f(t) = \infty$ such that
$$
\phi (t xx^t) = f(t) yy^t
$$
for every real $t$, $0 \le t \le p$, leads to a contradicition in a similar way.

It remains to consider the case that $\phi (X) = S_p XS_{p}^t$, $X \in [0, pI]$, for some invertible $n \times n$ real matrix $S_p$. But then it follows trivially that for every integer $n$, $n >p$, there exists an  invertible $n \times n$ real matrix $S_n$ such that $\phi (X) = S_n XS_{n}^t$, $X \in [0, nI]$. In particular, for any pair of integers $m,n >p$ and every unit vector $x$ we have 
$$
S_n xx^t S_{n}^t = S_m xx^t S_{m}^t
$$
which yields that $S_n x = t_x S_m x$ for some real number $t_x$. It can be easily verified that $t_x$ is independent of $x$. Thus, we have $S_n = t_{n,m} S_m$ for some nonzero real number $t_{n,m}$. From $S_n I S_{n}^t = S_m S_{m}^t$ we conclude that $ t_{n,m} = \pm 1$, and  by absorbing the constant we may assume with no loss of generality that $S_n \equiv S$ is independent of positive integer $n$, $n > p$. Hence, $\phi (X) = SXS^{-1}$ for every $X \in [0, \infty)$. Thus, we have the following statement.

\begin{theorem}
Let $\phi : [0, \infty ) \to [0, \infty)$ be an order automorphism. Then there exists a real invertible $n \times n$ matrix $T$ such that $\phi (X) = TXT^t$, $X \in [0, \infty)$.
\end{theorem}

In the next step we will characterize order automorphisms of $(0, \infty)$. So, let $\phi : (0, \infty ) \to (0, \infty)$ be an order automorphism. We want to show that there exists a real invertible $n \times n$ matrix $T$ such that $\phi (X) = TXT^t$, $X \in (0, \infty)$. Let $A \in (0 , \infty)$. Then the map $\phi_A : [0, \infty) \to [0 , \infty)$ given by
$$
\phi_A (X) = \phi ( X + A) - \phi (A), \ \ \  X \in [0, \infty),
$$
is obviously an order automorphism of $[0, \infty)$, and therefore we can find an invertible $n \times n$ matrix $T_A$ such that
$$
\phi_A (X) = T_A X T_{A}^t , \ \ \  X \in [0, \infty).
$$
For symmetric matrices $A,B$ satisfying $0 < B \le A$ and for every $X \ge 0$ we have
$$
T_A X T_{A}^t + \phi (A) = \phi (X+A) = \phi ((X + A-B) + B) = T_B (X + A - B)T_{B}^t + \phi (B) $$ $$=   T_B X T_{B}^t +  T_B (A - B)T_{B}^t + \phi (B).
$$
Choosing $X=0$ we see that
\begin{equation}\label{backtolj1}
\phi (A) - \phi (B) =   T_B (A - B)T_{B}^t.
\end{equation}
It follows that 
\begin{equation}\label{mrki}
T_A X T_{A}^t  = T_B X T_{B}^t , \ \ \ X \in [0, \infty).
\end{equation}
 For every nonzero vector $x$ the matrix $xx^t$ is positive. From (\ref{mrki}) we get that for every $x \in \mathbb{R}^n$ there exists a real number $t_x$ such that  $T_A x = t_x T_B x$. As before we see that $T_A = t T_B$ for some nonzero real number $t$. Inserting $X = I$ in (\ref{mrki}) we conclude that $\| T_A \| = \| T_B \|$. Hence, $t= \pm 1$, and by absorbing the constant we may assume with no loss of generality that $T_A \equiv T$ is independent of $A$, $A \in (0, \infty)$. Replacing the map $\phi$ by the map $X \mapsto T^{-1} \phi (X) (T^t)^{-1}$, $X \in (0, \infty)$, and using 
(\ref{backtolj1}) we can  assume with no loss of generality that 
$$
\phi (A) - \phi (B) = A-B 
$$
holds true for every pair  $A,B$ satisfying $0 < B \le A$. Applying the fact that $\phi$ is an order automorphism of $(0,I)$ we easily conclude that for any decreasing sequence $(B_n) \subset ( 0, \infty)$ with $\lim B_n = 0$ we have $\lim \phi (B_n) =0$. It follows that $\phi (A) = A$ for every $A \in (0, \infty)$, as desired. Hence, we have the following result.

\begin{theorem}
Let $\phi : (0, \infty ) \to (0, \infty)$ be an order automorphism. Then there exists a real invertible $n \times n$ matrix $T$ such that $\phi (X) = TXT^t$, $X \in (0, \infty)$.
\end{theorem}

It remains to consider order automorphisms of $(-\infty, \infty )$. Assume that $\phi: S_n \to S_n$ is an order automorphism. Then the map $X \mapsto \phi (X) - \phi (0)$, $X \in S_n$, is an order automorphism that maps $0$ to $0$. Consequently, its restriction to $[0, \infty)$ is an automorphism of $[0, \infty)$ onto itself. It follows that there exists an invertible real matrix $T$ such that $\phi (X) = TXT^\ast + S$, $X \in [0,I)$. Here, $S$ denotes $S= \phi (0)$. In a similar way as in the previous paragarph we can now prove the next theorem.

\begin{theorem}
Let $\phi : S_n \to S_n$ be an order automorphism. Then there exist an invertible $n \times n$ real matrix $T$ and a symmetric $n \times n$ real matrix $S$ such that
$$
\phi (X) = TXT^t + S, \ \ \ X \in S_n.
$$
\end{theorem}


\begin{thebibliography}{99}

\bibitem{BG}P. Busch and S.P. Gudder, Effects as functions of projective Hilbert space, {\em Lett. Math. Phys.} {\bf 47} (1999), 329--337.

\bibitem{BLPY}P. Busch, P. Lahti, J.-P. Pellonp\" a\" a, and K. Ylinen, {\em Quantum measurement}, Springer, 2016.

\bibitem{Dav}E.B. Davies, {\em Quantum theory of open systems}, Academic Press, 1976.

\bibitem{DoM}G. Dolinar and L. Moln\' ar, Sequential endomorphisms of finite-dimensional Hilbert space effect algebras, {\em J. Phys. A} {\bf 45} (2012), 065207, 11 pp.


\bibitem{Fa}C.-A. Faure, An elementary proof of the fundamental theorem of projective
geometry, {\em Geom. Dedicata} {\bf 90} (2002), 145--151.


\bibitem{GeS}G. P. Geh\' er and P. \v Semrl, Coexistency on Hilbert space effect algebras and a characterisation of its symmetry transformations, 
{\em Commun. Math. Phys.} {\bf 379}  (2020), 1077--1112. 


\bibitem{Har}L.A. Harris, Bounded symmetric homogeneous domains in infinite dimensional spaces,  {\em Lecture Notes in Math.} {\bf 364} (1974), 13--40.

\bibitem{HKX}J. Hou, K. He, and X. Qi, Characterizing sequential isomorphisms on Hilbert-space-effect algebras, {\em J. Phys. A} {\bf 43} (2010), 315206, 10pp.


\bibitem{Kraus}K. Kraus, {\em States, Effects, and Operations},  Lecture Notes in Physics \textbf{190}, Springer-Verlag, 1983.


\bibitem{LudI}G. Ludwig, {\em Foundations of Quantum Mechanics, Vol. I}, (Translated from the German by Carl A. Hein), Springer-Verlag, 1983.

\bibitem{LudII}G. Ludwig, {\em Foundations of Quantum Mechanics, Vol. II}, (Translated from the German by Carl A. Hein), Springer-Verlag, 1985.


\bibitem{Moln1}L. Moln\' ar, On some automorphisms of the set of effects on Hilbert space, {\em Lett. Math. Phys.} {\bf 51} (2000), 37--45.


\bibitem{Mo1}L. Moln\' ar,  Order-automorphisms of the set of bounded observables, {\em J. Math. Phys.} {\bf 42} (2001), 5904 - 5909. 

\bibitem{Mo2}L. Moln\' ar,  Characterizations of the automorphisms of Hilbert space effect algebras, {\em Comm. Math. Phys.} {\bf 223} (2001),  437 - 450. 

\bibitem{Mol}L. Moln\' ar, Order automorphisms on positive definite operators and a few applications,
{\em Linear Algebra Appl.} {\bf 434} (2011), 2158 - 2169.

\bibitem{Mo3}L. Moln\' ar, On the nonexistence of order isomorphisms between the sets of all self-adjoint and all positive definite operators, 
{\em Abstr. Appl. Anal.} (2015), Art. ID 434020, 6 pp. 

\bibitem{MoK}L. Moln\' ar and E. Kov\' acs, An extension of a characterization of the automorphisms
of Hilbert space effect algebras, {\em Rep. Math. Phys.} {\bf 52} (2003), 141 - 149.

\bibitem{MoN}L. Moln\' ar and G. Nagy, Spectral order automorphisms on Hilbert space effects and observables: the 2-dimensional case, {\em Lett. Math. Phys.} {\bf 106} (2016), 535--544.



\bibitem{MoP}L. Moln\' ar and Z. P\' ales, $\perp$-order automorphisms of Hilbert space effect algebras:
The two-dimensional case, {\em J. Math. Phys.} {\bf 42} (2001), 1907--1912.

\bibitem{MoS}L. Moln\' ar and P. \v Semrl, Spectral order automorphisms of the spaces of Hilbert space effects and
observables, {\em Lett. Math. Phys.} {\bf 80} (2007), 239--255.

\bibitem{MS}M. Mori and P. \v Semrl, Loewner's theorem for maps on operator domains, {\em Canad. J. Math.} {\bf 75}  (2023), 912--944.

\bibitem{MS1}M. Mori and P. \v Semrl, Optimal version of the fundamental theorem of chronogeometry, {\em Adv. Math.} {\bf 480} (2025), 110528, 85pp.

\bibitem{RW}M. Roelands and M. Wortel, Order isomorphisms on order intervals of atomic JBW-algebras, {\em Integral Equations Operator Theory} {\bf 92} (2020), Paper No. 33, 20pp.

\bibitem{SeT}Z. Sebesty\' en and Z. Tarcsay, On the square root of a positive selfadjoint operator,
{\em Period. Math. Hung.} {\bf 75}  (2017), 268--272. 

\bibitem{Se0}P. \v Semrl, Comparability preserving maps on Hilbert space effect algebras, {\em Comm. Math. Phys.} {\bf 313} (2012), 375--384.


\bibitem{Se1}P. \v Semrl, Symmetries of Hilbert space effect algebras, {\em J. London Math. Soc.}
{\bf  88}  (2013), 417 - 436.

\bibitem{Se4}P. \v Semrl,  Automorphisms of Hilbert space effect algebras,
{\em J. Phys.  A}
{\bf 48}
(2015), 195301,
18pp.

\bibitem{Se6}P. \v Semrl, Order isomorphisms of operator intervals,  {\em Integral Equations Operator Theory} {\bf 89} (2017), 1--42.

\bibitem{VIR}Hendrik van Imhoff and Mark Roelands, Order isomorphisms between cones of JB-algebras, {\em Studia Math.} {\bf 254} (2020), 179--198.
\end{thebibliography}
\end{document}